\documentclass[11pt]{article}

\usepackage[numbers,sort&compress]{natbib}
\usepackage{hyperref}
\usepackage{amsmath}
\usepackage{amssymb}
\usepackage{amsfonts}
\usepackage{amsthm}

\def\init{\setcounter{equation}{0}}

\newtheorem{theorem}{Theorem}[section]

\newtheorem{proposition}[theorem]{Proposition}

\newtheorem{lemma}[theorem]{Lemma}

\def\zz{\mathbb{Z}}

\def\nn{\mathbb{N}}
\def\rr{\mathbb{R}}
\def\fav{\hbox{Fav}}

\begin{document}

\title{The Favard length of product Cantor sets}
\author{Izabella {\L}aba and Kelan Zhai}

\maketitle

\begin{abstract}
Nazarov, Peres and Volberg proved recently in \cite{NPV} that the Favard length of the
$n$-th iteration of the four-corner Cantor set is bounded from above 
by $n^{-c}$ for an appropriate $c$.  We generalize this result to all product Cantor sets
whose projection in some direction has positive 1-dimensional measure.
\end{abstract}


\section{Introduction}\label{intro}

\init


Suppose that $A,B\subset \{0,1,\dots,K-1\}$, $|A|\cdot|B|=K$, and that both
$A$ and $B$ have cardinalities greater than 1.  We consider the 
Cantor set obtained by dividing the unit square in $\rr^2$ into $K^2$ congruent squares, keeping
those squares whose bottom left vertex is in $K^{-2}(A\times B)$ and deleting
the rest, then iterating the procedure. We use $E_n$ to denote the set obtained
after $n$ iterations, consisting of $K^n$ squares of sidelength $K^{-n}$.  Let also
$$
E_\infty=\bigcap_{n=1}^\infty E_n.
$$
Then $E_\infty$ is a closed set of Hausdorff dimension 1; moreover, it has finite and non-zero 
1-dimensional Hausdorff measure. It is furthermore a product of two Cantor sets $A_\infty$ and
$B_\infty$ of dimension $\alpha=\frac{\log |A|}{\log K}$ and 
$\beta=\frac{\log |B|}{\log K}$, respectively.  By our assumptions on $A,B$, 
\begin{equation}\label{e-gamma}
\gamma:=\min(\alpha,\beta)> 0.
\end{equation}  
It follows from a theorem of Besicovitch (see e.g. \cite{Falc})
that $E_\infty$ is purely unrectifiable,
hence almost every linear projection of $E_\infty$
has measure 0.  More precisely, let 
$\pi_\theta$ be the orthogonal projection in the direction $\theta$:
$$\pi_\theta(x,y)=x\cos\theta+y\sin\theta,$$
and define the {\em Favard length} of a set $\Omega\subset \rr^2$ by
$$\fav(\Omega)=\int_0^\pi |\pi_\theta(\Omega)|d\theta.$$
The Besicovitch projection theorem states that if $\Omega$ is purely unrectifiable, then
$\fav(\Omega)=0$.

We will be interested in quantitative versions of this statement.  By the dominated
convergence theorem, we have
\begin{equation}\label{ff-e1}
\lim_{\delta\to 0}\fav(\Omega^\delta)=0,
\end{equation}
where $\Omega^\delta$ denotes the $\delta$-neighbourhood of $\Omega$.
The question of interest concerns the rate of decay in (\ref{ff-e1}).  
Mattila \cite{Mattila} proved that for a very general class of 1-dimensional sets
$\Omega\subset \rr^2$ we have $\fav(\Omega^\delta)\geq C(\log(1/\delta))^{-1}$.
This includes as a special case the estimate
$$\fav(E_n)\geq Cn^{-1}$$
for the sets defined above.

Peres and Solomyak \cite{PS1} obtained the
first quantitative upper bound in (\ref{ff-e1}) for $E_n$ defined as above,
proving that
$$\fav(E_n)\leq C\exp(-C\log^*n),$$
where $\log^*n$ denotes the number of times that one must iterate the $\log$ function to get
from $n$ to 1.  Their result holds in fact for more general self-similar sets obeying the
``strong separation condition"; this includes for instance 1-dimensional Cantor type sets
in $\rr^2$ which are not product sets.
More recently, Tao \cite{Tao} proved
a quantitative version of (\ref{ff-e1}) under very general conditions on $X$. His 
upper bound is very weak (with convergence much slower that $\exp(-C\log^*n)$) 
and it is not clear whether this can be improved without additional assumptions.

The special case of the 4-corner Cantor
set ${\cal C}_n$ (defined as above with $K=4$, $A=B=\{0,3\}$) 
has been investigated in more detail.
In a recent breakthrough, Nazarov, Peres and Volberg \cite{NPV} proved
that 
$$\fav({\cal C}_n)\leq C_\epsilon n^{-\frac{1}{6}+\epsilon}$$
for any $\epsilon>0$.  
In the converse direction, Bateman and Volberg \cite{BV} proved that
$$\fav({\cal C}_n)\geq C n^{-1}\log n,$$
showing that at least in this case Mattila's lower bound is not optimal.  
However, Peres and Solomyak prove in \cite{PS1} that for ``random 4-corner sets"
${\cal R}_n$ we do have $\fav({\cal R}_n)\leq Cn^{-1}$. 

Our main result extends the estimate of \cite{NPV} to a much more general class of
product Cantor sets satisfying a natural geometrical condition.

\begin{theorem}\label{tiling-thm}
Let $E_n$ be defined as above.  Assume moreover that there is a direction $\theta_0$
such that the projection $\pi_{\theta_0}(E_\infty)$ has positive 1-dimensional
measure.
Then there is a $p\in(6,\infty)$ 
(depending on the choice of $K,A,B$) such that
$$
\fav(E_n)\leq C n^{-1/p}.
$$ 
\end{theorem}

The explicit range of $p$ in Theorem \ref{tiling-thm} is as follows.  By the results
of Kenyon \cite{Ke} and Lagarias-Wang \cite{LW1}, if $\pi_{\theta_0}(E_\infty)$ has 
positive 1-dimensional measure, $t_0:=\tan\theta_0$ must be rational.  Write 
$t_0=q/r$ in lowest terms.  Then the conclusion of our theorem holds for
all $p>6+4(1+q+r)\gamma^{-1}$, where $\gamma$ was defined in (\ref{e-gamma}). 
We do not
expect this range of $p$ to be optimal, in particular it may be possible to improve it
by proving the analogue of the ``stronger" harmonic-analytic estimate of
\cite{NPV}, but it is unlikely that an optimal estimate could be obtained
without developing substantially new methods.

Note that the 4-corner Cantor set has the property in question, with $\theta_0=\tan^{-1}
(2)$. This property also turns out to be  
crucial for the proof of the lower bound in \cite{BV}.
We expect this to be the ``worst" case in the sense that if all projections
of a 1-dimensional Cantor set $E_\infty$ have 1-dimensional measure 0,
then Mattila's lower bound should be sharp, i.e. 
$\fav(E_n)\leq Cn^{-1}$.  However, we do not know how to prove it at this point.

The general scheme of our proof follows that of \cite{NPV}.  Furthermore, many arguments
of \cite{NPV} can be adapted to our setting with only very minor changes,
including the entire self-similarity argument in \cite[Section 3]{NPV}
and some of the harmonic-analytic estimates in \cite[Section 2]{NPV}.
We only sketch those parts very briefly.  

The crucial difficulty in extending the bound of \cite{NPV} to more general sets is that
the main harmonic-analytic inequality in \cite[Section 2]{NPV}
(Proposition \ref{dd-prop1} in our paper) is based in part on an explicit evaluation of a 
certain ``Riesz product" in a closed form.  This calculation is specific to the
case of the 4-corner Cantor set and does not extend to the general case.

The starting point for our approach is the connection between tiling questions 
and projections of self-similar sets, first pointed out by Kenyon \cite{Ke}
and then developed further by Lagarias and Wang \cite{LW1}.
Kenyon proved that the projection of a 1-dimensional 
``Sierpi\'nski gasket" in a direction $\theta$ with $t=\tan\theta$ has positive length
if and only if $t=p/q$ in lowest terms, where the set $\{1,p,q\}$ satisfies
an appropriate tiling condition.  Lagarias and Wang \cite{LW1} extend this argument 
to prove a more general result which implies in particular the rationality of the
slope $t_0$ as stated above.  

To the best of our knowledge, tiling conditions have
not been used previously in connection with the quantitative Favard length problem
under consideration. 
It turns out, though, that the analysis in \cite{Ke}, \cite{LW1}
can indeed be pushed further to get additional information about the finite iterations of 
our Cantor set.  The relevant result of Kenyon and Lagarias-Wang is that 
$\pi_{\theta_0}(E_\infty)$ tiles $\rr$ by translations.  The same is not necessarily true
for the finite iterations $\pi_{\theta_0}(E_n)$; however, with a little more care
we can identify an appropriate tiling of $\rr$ associated with such iterations.
Analyzing the structure of this tiling in more detail, and
rewriting this information in terms of generating functions, we are able to produce
a good substitute for the trigonometric identities used in \cite{NPV}.

The paper is organized as follows.  We establish the notation and various conventions in
Section \ref{sec-prel}.  In Section \ref{sec-outline} we give the outline of 
the proof; since this is almost identical to that in \cite{NPV} except for the 
exponents, we only sketch it very briefly.  In Section \ref{sec-prop1} we focus on
the main harmonic-analytic inequality, Proposition \ref{dd-prop1},
and explain the main steps in its proof.  Of the two main estimates in the proof
of Proposition \ref{dd-prop1}, one (Lemma \ref{ww.lemma1}) is obtained by the 
same argument as in \cite{NPV}.  The second main estimate,
Lemma \ref{zz.lemma1}, is our main new contribution and 
relies on our tiling argument.  The latter is developed in Subsections \ref{tiling-sub1} 
and \ref{tiling-sub2}.  In Subsection \ref{tiling-sub3}, we combine it with the 
periodicity argument of \cite{NPV} to prove Lemma \ref{zz.lemma1}.

\bigskip\noindent
{\bf Acknowledgement.}  The authors are grateful to Alexander Volberg for helpful
discussions regarding the paper \cite{NPV}. Part of this research was done
while both authors were long-term visitors at the Fields Institute in Toronto.  
This work was supported by the Fields Institute and an NSERC Discovery Grant.


\section{Preliminaries}\label{sec-prel}
\init


Throughout this paper 
we will use $C,C',c$, etc.\ to denote various constants which may depend
on $K$ and the choice of $A$ and $B$, and may change from line to line, but always remain
independent of $N,n,m$.  We will sometimes write $P\gtrsim Q$ instead of $P\geq CQ$.
$P\approx Q$ will mean that $P\lesssim Q$ and $Q\lesssim P$.  If $G,H\subset\rr$ and
$\lambda\in\rr$, we write $\lambda G=\{\lambda g:\ g\in G\}$ and
$G+H=\{g+h:\ g\in G,\ h\in H\}$.

It suffices to consider $\theta\in[0,\pi/4]$.  We let $t=\tan \theta$, so that
$0\leq t\leq 1$.  
Let $A_n=K^{-n}A=\{K^{-n}a:\ a\in A\}\subset[0,K^{1-n}]$.  We also define
$$
A^n=A_1+A_2+\dots+A_n,
$$
$$
A^m_n=A_{n+1}+A_{n+2}+\dots+A_{m}=K^{-n}A^{m-n},\ m>n.
$$
Thus $|A^n|=|A|^n$, $|A^m_n|=|A|^{m-n}$, and $A^m=A^m_n+A^n$ for $m>n$.  The sets
$B_n$, $B^n$, $B^m_n$ are defined similarly.  Then
$$
E_n=(A^n\times B^n)+[0,K^{-n}]^2.
$$
Define
$$
\nu^n[t]=K^{-n}\sum_{a\in A^n,b\in B^n}\delta_{a+tb},
$$
$$
\nu^m_n[t]=K^{-m+n}\sum_{a\in A^m_n,b\in B^m_n}\delta_{a+tb},\ m>n,
$$
so that $\nu^m[t]=\nu^m_n[t]*\nu^n[t]$.  

We also need a counting function defined as follows.  Let $\phi(x)$ be the characteristic
function of the interval $[0,1]$, and let $\chi=\widehat{\phi}$.
We will also use the notation
$$\phi_n(x)=K^n\phi(K^n x),$$
then $\widehat{\phi_n}=\chi(K^{-n}\xi)=:\chi_n(\xi)$.  

Finally, we define our counting function
$$
F^n[t]=\nu^n[t]*\phi_n.
$$
This is almost identical (up to a mild rescaling of $\phi$) to the counting function used in
\cite{NPV}, which can be generalized to our setting as 
$
f^n[t]=\nu^n[t]*K^n{\bf 1}_{[0,(1+t)K^{-n}]},
$
so that $f^n[t](x)$ counts the number of squares that $\pi_\theta$ projects to
$x\cos\theta$.

For simplicity of notation, we will 
often suppress the dependence on $t$ and write
$F^m[t]=F^m$, $\nu^m[t]=\nu^m$, $\nu^m_n[t]=\nu^m_n$.  
We then have
$$
\widehat{F^m}=\widehat{\nu^m}\cdot\chi_m,\ $$
$$
\widehat{\nu^m}(\xi)=K^{-m}\sum_{a\in A^m,b\in B^m}e^{-2\pi i (a+tb)\xi}.
$$


\section{Outline of the proof}\label{sec-outline}
\init


Let $\theta_0, t_0, q,r$, be as in Theorem \ref{tiling-thm} and the comment immediately
following it.  Define
$$\sigma=(1+r+q)\gamma^{-1}.$$

\begin{proposition}\label{dd-prop1}
Let
$$
X_\lambda^N=\{t\in [0,1]: \ \max_{1\leq n\leq N}\int |F^n[t]|^2dx\leq\lambda\}.
$$
Then for all $\lambda\leq CK^{N/10}$ we have
\begin{equation}\label{dd-e20}
|X_\lambda^N|\leq C(\log\lambda)^{1/4\sigma}\lambda^{1+\sigma^{-1}} N^{-1/4\sigma}.
\end{equation}
\end{proposition}

Assuming Proposition \ref{dd-prop1}, 
we complete the proof of Theorem \ref{tiling-thm} as follows.  Since the argument here is
identical to that in \cite{NPV}, except that the exponents are modified to 
match our setting, we only give a very brief outline.  

Let $m(\lambda)=|\{x:\ G^N(x)\geq \lambda\}|$, where $G^N(x)=\max_{1\leq n\leq N}
F^N(x)$.  Following the argument in \cite{NPV}, one can prove that
$$
m(2a\lambda\lambda')\leq C\lambda m(\lambda)m(\lambda'),
$$
where $a$ denotes the maximal number of squares in $E_1$ that a straight line can intersect.
By induction,
$$
m(\lambda(2a\lambda)^j)\leq (C\lambda m(\lambda))^j m(\lambda).
$$
Let $T_N=\{t:\ m(\lambda)\leq c\lambda^{-3}\}$, with $c$ small enough.  As in
\cite{NPV}, we verify that for $t\in T_N$ we have
$$
\int (F^N[t])^2dx\leq 2\lambda.
$$
By Proposition \ref{dd-prop1}, we have 
\begin{equation}\label{dd-e71}
|T_N|\leq C \lambda^{1+\sigma^{-1}+\epsilon}
N^{-1/4\sigma}
\end{equation}
for any $\epsilon>0$, assuming that
\begin{equation}\label{dd-e70}
\lambda\leq CK^{N/10}.
\end{equation}
Here and in what follows, the various constants 
may depend on $\epsilon$ but we do not display that dependence. 

Following an argument of \cite{NPV} again, we see that for $t\notin T_N$ we have
$|\pi_\theta E_{RN}|\leq C\lambda^{-1}$ whenever $R\geq C'\lambda^2\log\lambda$ with
$C'$ large enough.  Let $s=C\lambda^{-1}$ and $M=RN$.  Then 
$$
\{\theta:\ |\pi_\theta E_{M}|\geq s\}\subset T_{M(C'\lambda^2\log\lambda)^{-1}},
$$
and the latter set has measure bounded by 
$$
C \lambda^{1+\sigma^{-1}+\epsilon}(M(C'\lambda^2\log\lambda)^{-1})^{-1/4\sigma}
=C(M\, s^{6+4\sigma +\epsilon'})^{-1/4\sigma}.
$$
Hence for $p=6+4\sigma+\epsilon'$ we have
\begin{align*}
\int_0^\pi |\pi_\theta (E_{M})|d\theta
&=\int_0^\infty |\{\theta:\ |\pi_\theta (E_{M})|\geq s\}|ds
\\
&\lesssim \int_0^{M^{-1/p}}1ds +\int_{M^{-1/p}}^\infty (Ms^p)^{-1/4\sigma}ds
\\
&\lesssim M^{-1/p}+M^{-1/4\sigma}M^{-\frac{1}{p}(1-\frac{p}{4\sigma})}
=2M^{-1/p}.\\
\end{align*}
We used here that if $s=C\lambda^{-1}\geq M^{-1/p}$, then $\lambda$ obeys 
(\ref{dd-e70}), hence (\ref{dd-e71}) applies.


\section{Proof of Proposition \ref{dd-prop1}}\label{sec-prop1}
\init


We now give a proof of Proposition \ref{dd-prop1} modulo several auxiliary results 
whose proofs are deferred
to later sections.  Fix $\lambda$ and $N$, and write $X_\lambda^N=X$.  
We also choose an integer $m$ so that 
\begin{equation}\label{dd-e21}
K^{m}\approx C\lambda
\end{equation}
with $C$ sufficiently large.  

Using Plancherel's theorem and that $|\chi_N|\geq c$ for $|\xi|\leq K^m$, we get that
$$
\int_1^{K^N} |\widehat{\nu^N}|^2d\xi
\leq c^{-1}\int_1^{K^N}|\widehat{F^N}|^2d\xi
\leq c^{-1}\lambda.
$$
The first lemma follows by easy pigeonholing.

\begin{lemma}\label{dd-lemma2}
For $1\leq m,n\leq N/10$, define
$$
I=I[t]=\int_{K^n}^{K^{m+n}} |\widehat{\nu^N}|^2\,d\xi.
$$
Then for each $1\leq m\leq N/10$ there is a $1\leq n\leq N/10$ and a
set $X'\subset X$ with $|X'|\geq \frac{1}{2}|X|$ such that for all
$t\in X'$ we have 
\begin{equation}\label{dd-e0}
I\leq \frac{C\lambda m}{N}.
\end{equation}
\end{lemma}

The key to the proof of Proposition \ref{dd-prop1} is the following estimate.

\begin{proposition}\label{prop-i}
There is a $t\in X'$ such that
$$
I\gtrsim \lambda^{-3-4\sigma}|X|^{4\sigma}.
$$

\end{proposition}

The last inequality can only be compatible with (\ref{dd-e0}) if 
$\frac{\lambda m}{N}\gtrsim \lambda^{-3-4\sigma}
|X|^{4\sigma}$, 
i.e.
$$
|X|\leq Cm^{1/4\sigma}\lambda^{1+\sigma^{-1}} N^{-1/4\sigma}.
$$
Note that by (\ref{dd-e21}), $m\approx\log\lambda$.  Hence we have proved 
Proposition \ref{dd-prop1}.
\qed

It remains to prove Proposition \ref{prop-i}.  We write
\begin{align*}
I&= \int_{K^n}^{K^{m+n}} |\widehat{\nu^N}|^2\,d\xi\\
&\approx \int_{K^n}^{K^{m+n}} |\widehat{\nu^{m+n}}|^2\,d\xi\\
&=\int_{K^n}^{K^{m+n}} |\widehat{\nu^{m+n}_n}|^2\,|\widehat{\nu^n}|^2\,d\xi.
\end{align*}

The main two estimates are as follows.

\begin{lemma}\label{ww.lemma1}
Let
$$
I_1=\int_{K^n}^{K^{m+n}}|\widehat{\nu^n}|^2\,d\xi.
$$
Then for any $t\in X'$ we have
\begin{equation}\label{bb.e100}
I_1\gtrsim K^m.
\end{equation}
\end{lemma}

\begin{lemma}\label{zz.lemma1}
Let 
$$
I_2=\frac{1}{|X'|}\int_{X'}
\int_{[K^n,K^{m+n}]\cap Z_\delta}
|\widehat{\nu^n}|^2d\xi\,dt,
$$
where $Z_\delta=\{\xi:\ |\widehat{\nu_n^{m+n}}|\leq K^{-2m}\delta^2\}$. 
Let $c>0$ be a small constant, and let
\begin{equation}\label{zz.e66}
\delta=c_0 \lambda^{-\sigma}|X|^{\sigma},
\end{equation}
where $\sigma$ is defined right after Theorem \ref{tiling-thm}.  Then
\begin{equation}\label{zz.e0}
I_2\leq cK^{m}.
\end{equation}
\end{lemma}

Assuming Lemmas \ref{ww.lemma1} and \ref{zz.lemma1}, we write
\begin{align*}
\frac{1}{|X'|}\int_{X'}I(t)dt
&=\frac{1}{|X'|}\int_{X'}
\int_{K^n}^{K^{m+n}}|\widehat{\nu_n^{m+n}}|^2
|\widehat{\nu^n}|^2d\xi\,dt
\\
&\geq K^{-4m}\delta^4 \frac{1}{|X'|}\int_{X'}
\int_{[K^n,K^{m+n}]\setminus Z_\delta}|\widehat{\nu^n}|^2d\xi\,dt
\\
&= K^{-4m}\delta^4 \Big(\frac{1}{|X'|}\int_{X'}I_1(t)dt
-I_2\Big).
\\
\end{align*}
By (\ref{bb.e100}) and (\ref{zz.e0}), $\int_{X'}I(t)dt\geq c_0K^{-3m}\delta^4
|X'|$.  Thus there is at least one $t\in X'$ such that
$$
I(t)\geq c_0 K^{-3m}\delta^4\approx \lambda^{-3}\lambda^{-4\sigma}|X|^{4\sigma},
$$
which ends the proof of the proposition.
\qed

Lemma \ref{ww.lemma1} is proved exactly as in \cite[pp. 5-7]{NPV} (the ``$P_2$ estimate").
Lemma \ref{zz.lemma1} is our main new contribution.  We will now proceed to its proof.


\section{The tiling argument}\label{sec-tiling}
\init


\subsection{A review of basic definitions}\label{tiling-sub1}

Our proof of Lemma \ref{zz.lemma1} will rely on various results on tilings
of the real line and of the integers.  We first recall the main definitions.
Let $\Omega$ be a subset of $\rr$ of positive Lebesgue measure.  We will say that 
$\Omega$ {\em tiles $\rr$ by translations} if there is a set ${\cal T}$, henceforth
called the {\em translation set}, such that the sets $\Omega+\tau$, $\tau\in{\cal T}$,
are disjoint up to sets of measure zero and
$$
\bigcup_{\tau\in{\cal T}} (\Omega+\tau)=\rr.
$$
We will sometimes write $\Omega\oplus {\cal T}=\rr$.  Lagarias and Wang \cite{LW1}
prove that any tiling of $\rr$ by a set $\Omega$ of bounded diameter is periodic:
there is a $T\neq 0$ (the period of the tiling) such that $T+{\cal T}={\cal T}$.
They also obtain a description of the structure of all such tilings
which we will have cause to invoke in the sequel. 

We will also need a few basic definitions and facts from the theory of
integer tilings and factorization of cyclic
groups.  A set of integers $A$ is said to {\em tile the integers}
if there is a set $C\subset\zz$ such that every integer $n$ can be written
in a unique way as $n=a+c$ with $a\in A$ and $c\in C$.  It is well known (see
\cite{New}) that any tiling of $\zz$ by a finite set $A$ must be periodic:
$C=B+M\zz$ for some finite set $B\subset \zz$ such that $|A|\,|B|=M$.  We
then write 
\begin{equation}\label{dd-e36}
A\oplus B=\zz_M
\end{equation}  
and say that $A$ and $B$ are {\em complementing sets} modulo $M$, or 
that $B$ is the {\em complement} of $A$. 
Without loss of generality, we may assume that $0\in A\cup B$. 
It is easy to verify that (\ref{dd-e36}) is then equivalent
to
\begin{equation}\label{e-tiling}
 A(x)B(x)=1+x+\dots+x^{M-1}\ \mod (x^M-1),
\end{equation}  
where we use $A(x)$ to denote the generating function of a set $A\subset\rr$:
$$
A(x)=\sum_{a\in A} x^a.
$$

The subject of integer tilings and factorization of abelian groups has been researched
quite extensively, see e.g. \cite{CM}, \cite{LW1}, \cite{LW2}, \cite{New}, \cite{S1},
\cite{Sz}, \cite{Tij}.  In this paper, we will rely on (\ref{e-tiling}) to provide
a key estimate in the next subsection.


\subsection{The approximate zero set estimate}\label{tiling-sub2}

Throughout this subsection we fix a value of $\theta_0$ so that 
$\pi_{\theta_0}(E_\infty)$ has positive 
1-dimensional measure.  By symmetry, we may assume that $0<\theta_0<\pi/2$. 
Let $t_0=\tan\theta_0$
and $T_\infty=\pi_{\theta_0}(E_\infty)$.

It was proved by Kenyon \cite{Ke} for a special case of a self-similar set, and Lagarias 
and Wang \cite{LW1} in
full generality, that $t_0$ must be rational and that $T_\infty$ tiles periodically the
real axis. For our purposes, we need additional information on the structure of the
tiling which can be extracted from \cite{Ke}, \cite{LW1} by analyzing their arguments
more carefully.  

Write $t_0=q/r$, $q,r\in\nn$. Note also that by self-similarity we have for any $n$
\begin{equation}\label{ss-e1}
T_\infty=\bigcup_{x\in A^n+tB^n}(x+K^{-n}T_\infty).
\end{equation}
An argument of \cite[Lemmas 3,4]{Ke} shows that
\begin{itemize}
	\item $T_\infty$ contains a line segment $J$, 
	\item the sets $x+K^{-n}T_\infty$ in (\ref{ss-e1}) are mutually disjoint,
	\item in particular, $J$ is covered by disjoint translates of the set $K^{-n}T_\infty$,
	\item this covering can be extended to a periodic tiling of $\rr$ by
	translates of $K^{-n}T_\infty$:
$$\bigcup_{x\in {\cal T}}(x+ K^{-n}T_\infty)=\rr,$$
where the sets $\{x+K^{-n}T_\infty\}_{x\in{\cal T}}$
are mutually disjoint.  
\end{itemize}

The set ${\cal T}$ is usually referred to as the {\em translation
set}. Here and below, by ``disjoint" we will mean ``disjoint up to a set of measure 0".

If $n$ is chosen large enough relative to the length of $J$, $J$ contains an entire period of
the tiling.  However, all $x\in {\cal T}$ such that $x+ K^{-n}T_\infty\subset J$
are of the form $a+tb$, $a\in A^n,b\in B^n$, so that $x\in r^{-1}K^{-n}\zz$.
It follows that ${\cal T}\subset  r^{-1}K^{-n}\zz$ and in particular the period
of the tiling is a number of the form $r^{-1}K^{-n}M$ with $M\in\zz$.

For convenience, we rescale everything by $rK^nM^{-1}$.  The rescaled tiling is 
1-periodic and has the form 
$$\bigcup_{x\in \Lambda+\zz}(x+ rM^{-1}T_\infty)=\rr,$$
where the sets $x+ rM^{-1}T_\infty$ are mutually disjoint and $\Lambda\subset
[0,1)\cap M^{-1}\zz$.  Using self-similarity as in (\ref{ss-e1}) and invoking
Kenyon's disjointness argument again, we rewrite this as
$$
\bigcup_{x\in \Lambda+\zz}(x+rM^{-1}A^m+qM^{-1}B^m+rM^{-1}K^{-m}T_\infty)=\rr
$$
or, equivalently, 
\begin{equation}\label{ss-e2}
\bigcup_{x\in \Lambda+rM^{-1}A^m+qM^{-1}B^m+\zz} (x+rM^{-1}K^{-m}T_\infty)=\rr,
\end{equation}
where the union is disjoint.
We consider (\ref{ss-e2}) as a tiling where $rM^{-1}K^{-m}T_\infty$ 
is the tile and $\Lambda+rM^{-1}A^m+qM^{-1}B^m+\zz$ is the translation set.  Note
that $\Lambda+rM^{-1}A^m+qM^{-1}B^m\subset K^{-m}M^{-1}\zz$.  

By \cite{LW1}, Theorem 3 and the argument in its proof on page 359, 
$MK^m(\Lambda+rM^{-1}A^m+qM^{-1}B^m)$ is a complementing set modulo $MK^m$.  Moreover,
it has a complement of the form $MK^{m}S^m$, where
$$
S^m=\{y\in K^{-m}M^{-1}\zz:\ y+w\in rM^{-1}K^{-m}T_\infty\}
$$
for some fixed $w\in(0,M^{-1}K^{-m})$.  In our case, the set $T_\infty$ has
diameter at most $1+t_0$.  Translating $S^m$ if necessary, we may assume that
\begin{equation}\label{ss-e3}
MK^m S^m\subset\{0, 1,\dots,r+q-1\}.
\end{equation}

The above statement on complementing sets can be reformulated in terms of generating 
functions:
\begin{equation}\label{ss-e4}
\begin{split}
&\Lambda(x^{MK^m})\cdot (rM^{-1}A^m)(x^{MK^m})\cdot (qM^{-1}B^m)(x^{MK^m})
\cdot S^m(x^{MK^m})
\\
&=\big(1+x+x^2+\dots+x^{MK^m-1}\big)\big(1+(x-1)P(x)\big),
\end{split}
\end{equation}
where, by (\ref{ss-e3}), $P(x)$ is a polynomial of degree at most $r+q-1$.
Changing variables and summing the geometric progression
on the right side, we rewrite (\ref{ss-e4}) as
\begin{equation}\label{ss-e5}
\begin{split}
&\Lambda(x)\cdot (rM^{-1}A^m)(x)\cdot (qM^{-1}B^m)(x)
\cdot S^m(x)
\\
&=\frac{x-1}{x^{1/MK^m}-1}
\big(1+(x^{1/MK^m}-1)P(x^{1/MK^m})\big),
\end{split}
\end{equation}

Furthermore, since $MK^m(\Lambda+rM^{-1}A^m+qM^{-1}B^m)$ and $MK^{m}S^m$ are 
complementing sets modulo $MK^m$, it follows that
$$
|\Lambda|\cdot|A^m|\cdot |B^m|\cdot|S^m|=MK^m,
$$
so that 
\begin{equation}\label{ss-e6}
|\Lambda|\cdot |S^m|=M,
\end{equation}
independently of $m$.

The main result of this subsection is the following lemma.

\begin{lemma}\label{ss-lemma1}
Suppose that $y\in [0,K^m]$ and
\begin{equation}\label{ss-e20}
|A^m(e^{-2\pi iy})||B|^m\leq\delta,
\end{equation}
and fix $\epsilon\in(0,1)$.  Then $y\in Y\cup (I_\delta\cap [0,K^m])$, where
$$I_\delta =rM^{-1}\zz+(-c\delta^{1-\epsilon},c\delta^{1-\epsilon}),$$
and $Y$ is the union of at most $r+q$ intervals of length at most 
$CK^m\delta^{\epsilon(r+q)^{-1}}$.  (All constants here depend on $A,B,M,q,r$.)
\end{lemma}

\begin{proof}
Let $x=e^{-2\pi iMy/r}$ and $\Delta=2M\delta$.  
We need to find the set where 
\begin{equation}\label{ss-e10}
|A^m(x^{rM^{-1}})|\leq \frac{\Delta}{2} M^{-1}|B|^{-m},
\end{equation}

Multiplying both sides of (\ref{ss-e5}) by $1-x^{1/MK^m}$, we get
\begin{equation}\label{ss-e7}
(1-x^{1/MK^m})\Lambda(x)\cdot A^m(x^{rM^{-1}})\cdot B^m(x^{qM^{-1}})
\cdot S^m(x)
=(1-x)Q(x^{1/MK^m}),
\end{equation}
where $Q(u)=1+(u-1)P(u)$ is a polynomial of degree at most $r+q$.
The left hand side of (\ref{ss-e7}) is bounded in absolute value by
$$
2|\Lambda|\cdot |A^m(x^{rM^{-1}})|\cdot |B|^m\cdot |S^m|
=2M|B|^m\,|A^m(x^{rM^{-1}})|,
$$
by (\ref{ss-e6}). It follows that if (\ref{ss-e10}) holds, then we must have
at least one of the following:
\begin{equation}\label{ss-e11}
|1-x|=|1- e^{-2\pi iMy/r}|\leq \Delta^{1-\epsilon},
\end{equation}
\begin{equation}\label{ss-e12}
|Q(x^{1/MK^m})|=|Q(e^{-2\pi iy/K^mr})|\leq\Delta^\epsilon.
\end{equation}
If (\ref{ss-e11}) holds, then 
$My/r\in \zz+(-\Delta^{1-\epsilon},\Delta^{1-\epsilon})$, so that
$y\in rM^{-1}\zz+(-c\delta^{1-\epsilon},c\delta^{1-\epsilon})$.
Suppose now that (\ref{ss-e12}) holds.  Clearly, this can only happen if 
$Q\neq\ $const.  
Let $z_j=r_je^{-2\pi i\sigma_j}$, $j=1,2,\dots,J$, be the distinct roots of $Q$;
we have $J\leq r+q$, and each $z_j$ has multiplicity at most $r+q$.
In order for (\ref{ss-e12}) to hold, we must have for some $j$
$$|x^{1/MK^m}-z_j|=|e^{-2\pi iy/K^mr}-r_je^{-2\pi i\sigma_j}|
\leq C\Delta^{\epsilon(r+q)^{-1}},
$$
hence in particular there is a $k\in\zz$ such that
$$
\frac{y}{rK^m}\in (\sigma_j+k
-C\Delta^{\epsilon(r+q)^{-1}},\sigma_j+k+C\Delta^{\epsilon(r+q)^{-1}}).
$$
But we are assuming that $y\in[0,K^m]$, so that $\frac{y}{rK^m}\in
[0,r^{-1}]$ and for each $j$ there is at most one $k=k(j)$ such that the above
interval intersects $[0,r^{-1}]$.  We let $Y$ be the union of these intervals
rescaled by $rK^m$.

\end{proof}


\subsection{Proof of Lemma \ref{zz.lemma1}}\label{tiling-sub3}

The proof below combines the periodicity argument of \cite{NPV} with the analysis
of the last two subsections.
Let $y=K^{-n}\xi$ and
\begin{equation}\label{gg-e2}
\delta=c_0 \lambda^{-\sigma}|X|^{\sigma},
\end{equation}
where $c_0$ is a sufficiently small constant and
\begin{equation}\label{gg-e10}
\sigma=(1+r+q)\gamma^{-1}.
\end{equation}
We need to prove that
\begin{equation}\label{gg-e1}
I_2\leq cK^m,
\end{equation}
where
\begin{equation}\label{zz.e101}
I_2=\frac{1}{|X'|}\int_{X'}\int_{[1,K^{m}]\cap K^{-n}Z_\delta}
K^n|\widehat{\nu^n}(K^{n}y)|^2 dy\,dt,
\end{equation}
$c$ is a small constant, and
$$Z_\delta=\{\xi:\ |\widehat{\nu_n^{m+n}}(\xi)|\leq K^{-2m}\delta^2\}.$$ 

We first use Lemma \ref{ss-lemma1} to give a description of the set $Z_\delta$.

\begin{lemma}\label{zz.lemma2}
Let $\epsilon\in(0,1)$ be a number to be fixed later.  We have
$$
K^{-n}Z_\delta\subset Y\cup t^{-1}Y\cup I_\delta\cup t^{-1}I_\delta,
$$
where 
$$I_\delta =rM^{-1}\zz+(-c\delta^{1-\epsilon},c\delta^{1-\epsilon}),$$
and $Y$ is the union of at most $r+q$ intervals of length at most 
$CK^m\delta^{\epsilon(r+q)^{-1}}$.
\end{lemma}

\begin{proof}
We have
\begin{align*}
\widehat{\nu_n^{m+n}}(K^n y)
&=K^{-m}\sum_{a\in A_n^{m+n},b\in B_n^{m+n}}e^{-2\pi i(a+tb)K^n y}
\\
&=K^{-m}\sum_{a\in A^{m},b\in B^{m}}e^{-2\pi i(a+tb)y}
\\
&=K^{-m}A^m(e^{-2\pi iy})B^m(e^{-2\pi ity}).
\\
\end{align*}

Suppose that $y\in K^{-n}Z_\delta$, then
\begin{align*}
\delta^2 &\geq K^{2m}|\widehat{\nu_n^{m+n}}(K^n y)|\\
&=K^{m}|A^m(e^{-2\pi iy})||B^m(e^{-2\pi ity})|\\
&=|A^m(e^{-2\pi iy})||B|^m\cdot |B^m(e^{-2\pi ity})||A|^m.\\
\end{align*}

Thus at least one of $|A^m(e^{-2\pi iy})||B|^m$ and $|B^m(e^{-2\pi ity})||A|^m$
is less than $\delta$.  The conclusion follows by applying 
Lemma \ref{ss-lemma1} to both terms.

\end{proof}

{\em Proof of (\ref{gg-e1})}.  
We estimate 
\begin{align*}
I_2 
&\lesssim \frac{K^n}{|X'|}\int_0^1 dt 
\int_{(I_{\delta}\cup t^{-1}I_\delta\cup Y\cup t^{-1}Y)\cap[1,K^m]}
|\widehat{\nu^n}(K^n y)|^2dy
\\
&= \frac{K^{-n}}{|X'|} \int_0^1 dt 
\int_{(I_{\delta}\cup t^{-1}I_\delta\cup Y\cup t^{-1}Y)\cap[1,K^m]}
|A^n(e^{-2\pi iK^ny})B^n(e^{-2\pi itK^ny})|^2 dy
\\
&\leq \frac{K^{-n}}{|X'|} \int_0^1 dt \int_{y\in (I_{\delta}\cup Y)\cap[1,K^m]}\dots dy
+ \frac{K^{-n}}{|X'|} \int_0^1 dt \int_{ty\in (I_\delta\cup Y)\cap[0,K^m]}\dots dy.
\\
\end{align*}
Changing variables $(y,t)\to (y,u)$, where $u=ty$, $du=ydt$, we estimate
the last line by
\begin{align*}
&\frac{K^{-n}}{|X'|}\Big(\int_{(I_{\delta}\cup Y)\cap[1,K^m]} 
|A^n(e^{-2\pi iK^ny} )|^2\frac{dy }{y}\Big)
\Big(\int_0^{K^m} |B^n(e^{-2\pi iK^n u})|^2du\Big)
\\
&+\frac{K^{-n}}{|X'|}\Big(\int_{(I_{\delta}\cup Y)\cap[1,K^m]} 
|B^n(e^{-2\pi iK^n u})|^2 du\Big)
\Big(\int_1^{K^m} |A^n(e^{-2\pi iK^ny})|^2\frac{dy}{y}\Big).
\\
\end{align*}
Note that $A^n(e^{-2\pi iK^ny})$ is $1$-periodic and that for any interval
$J$ of length 1 we have
\begin{equation}\label{zz.e105}
\int_J |A^n(e^{-2\pi iK^ny})|^2\,dy=\sum_{a,a'\in A^n}\int_J e^{2\pi i(a-a')K^ny}dy
=|A|^n,
\end{equation}
since all terms with $a\neq a'$ vanish.  Using also that 
$y\geq 1$ on the region of 
integration, and applying the same argument to $B^n$, we get that
$$ I_2\lesssim 2 \frac{K^{m-n}}{|X'|}\big( (I_3+I_4)|B|^n+(I_5+I_6)|A|^n\big),$$
where
\begin{equation*}
I_3=\int_{I_{\delta}\cap[1,K^m]} 
|A^n(e^{-2\pi iK^n y})|^2 dy,
\end{equation*}
\begin{equation*}
I_4=\int_Y |A^n(e^{-2\pi iK^n y})|^2 dy,
\end{equation*}
\begin{equation*}
I_5=\int_{I_{\delta}\cap[1,K^m]} 
|B^n(e^{-2\pi iK^n u})|^2 du,
\end{equation*}
\begin{equation*}
I_6=\int_Y |B^n(e^{-2\pi iK^n u})|^2 du,
\end{equation*}
Since $|A||B|=K$ and $|X'|\geq |X|/2$, (\ref{gg-e1}) will follow if we can prove that
\begin{equation}\label{zz.e2}
\max (I_3,I_4)\leq c|A|^n |X|, \ 
\max (I_5,I_6)\leq c|B|^n |X|, \ 
\end{equation}
for some small enough $c$.   We only consider $I_3$ and $I_4$, the case
of $I_5$ and $I_6$ being similar with $A$ and $B$ interchanged.

We fix 
$$\epsilon=\frac{r+q}{1+r+q}.$$
Assume first that 
\begin{equation}\label{gg-e100}
c\delta^{1-\epsilon}\geq 10K^{-n}
\end{equation} 
and start with $I_3$.
By (\ref{gg-e100}), we may choose $j$ so that $K^{j-n}\approx c\delta^{1-\epsilon}
=c\delta^{(1+r+q)^{-1}}$.
Write $A^n(x)=A^n_j(x)A^j(x)$ 
with the implicit constant small enough.  
Note that $A^j(e^{-2\pi iK^ny})$ is $K^{j-n}$-periodic,
hence for any interval $J$ of length $c\delta^{(1+r+q)^{-1}}$ we have
$$\int_J|A^j(e^{-2\pi iK^ny})|^2\,dy
\approx \delta^{(1+r+q)^{-1}} \int_0^1 |A^j(e^{-2\pi iK^ny})|^2\,dy
=\delta^{(1+r+q)^{-1}} |A|^j,$$
by the same argument as in (\ref{zz.e105}). 
Since the region of integration in $I_3$ is a union of about $K^m$ such intervals, and
since $|A^n_j(e^{2\pi i\xi})|\leq |A|^{n-j}$, it follows that
$$
I_3\lesssim |A|^{2(n-j)}\,K^m\delta^{(1+r+q)^{-1}}|A|^j
=|A|^{2n-j}\,K^m\,\delta^{(1+r+q)^{-1}}.
$$
Recall that $K^m\approx\lambda$.  By the choice of $j$, and using
that $|A|^{n-j}\leq K^{(n-j)(1-\gamma)}\lesssim
\delta^{(\gamma-1)(1+r+q)^{-1}}$, we get that
\begin{align*}
I_3
&\lesssim \lambda \delta^{(\gamma-1)(1+r+q)^{-1}}\delta^{(1+r+q)^{-1}}|A|^n\\
&=\lambda \delta^{\gamma(1+r+q)^{-1}}|A|^n\\
&=\lambda \delta^{1/\sigma}|A|^n,
\end{align*}
which is less than $c|A|^n|X|$ thanks to (\ref{gg-e2}).

Next, we consider $I_4$. Recall that $Y$ is the union of at most $r+q$ intervals 
of length at most 
$$
CK^m\delta^{\epsilon(r+q)^{-1}}=K^m\delta^{(1+r+q)^{-1}}
\approx K^m(K^{-m}|X|)^{1/\gamma}
=K^{m(1-\frac{1}{\gamma})}|X|^{1/\gamma}.
$$
In particular, this is less than 1.  We may therefore choose a $k$ such that
$K^m\delta^{(1+r+q)^{-1}}\approx K^{k-n}$ with a small enough implicit
constant.  Writing $A^n(x)=A^n_k(x)A^k(x)$ and arguing as in the case of $I_3$, 
we see that
\begin{align*}
I_4&\lesssim |A|^{2(n-k)}K^m\delta^{(1+r+q)^{-1}}|A|^k
\\
&\approx |A|^{n-k}|A|^n\lambda \delta^{1/(1+r+q)}|A|^k.\\
\end{align*}
We have $|A|^{n-k}\leq K^{(n-k)(1-\gamma)}\leq (K^m\delta^{(1+r+q)^{-1}})^{\gamma-1}$,
hence
\begin{align*}
I_4&\lesssim K^{m(\gamma-1)}\lambda |A|^n\delta^{\gamma(1+r+q)^{-1}}
\\
&\lesssim \lambda |A|^n\delta^{\gamma(1+r+q)^{-1}}
\\
&=\lambda |A|^n\delta^{1/\sigma}.
\end{align*}
Thus $I_4\leq c|A|^n|X|$ whenever (\ref{gg-e2}) holds with a sufficiently small constant.

It remains to consider the case when (\ref{gg-e100}) fails.  In this case, we bound $I_3$
by
$$
I_3\leq |I_\delta\cap[1,K^m]|\cdot |A|^{2n}\lesssim K^{m}\delta^{1-\epsilon}|A|^{2n}.
$$
It suffices to prove that this is less than $c|A|^n|X|$, or equivalently that
\begin{equation}\label{gg-e101}
C\delta^{1-\epsilon}|A|^n\leq K^{-m}|X|
\end{equation}
if the constant $c_0$ in (\ref{gg-e2}) has been chosen small enough depending on $C$.  Using that 
$\lambda\approx K^m$ and $|A|\leq K^{1-\gamma}$, we see that (\ref{gg-e101}) will follow
if 
$$C\delta^{1-\epsilon}K^{(1-\gamma)n}\leq \lambda^{-1}|X|=(c_0^{-1}\delta)^{1/\sigma},$$
where at the last step we used (\ref{gg-e2}). This
is equivalent to
$$
CK^{(1-\gamma) n}\leq c_0^{-1/\sigma}\delta^{\frac{1}{\sigma}-1+\epsilon}
=c_0^{-1/\sigma}\delta^{-\frac{1-\gamma}{1+r+q}}.
$$
Using that $1-\epsilon=1/(1+r+q)$, we rewrite this as 
$$c_0^{1/\sigma(1-\gamma)}\delta^{1-\epsilon}\leq C^{-1/(1-\gamma)}K^{-n}.$$
But this is a consequence of the failure of (\ref{gg-e100}) if $c_0$ was chosen small enough.  
The verification for $I_4$ is identical.


\noindent{\sc Department of Mathematics, University of British Columbia, Vancouver,
B.C. V6T 1Z2, Canada}

\noindent{\it ilaba@math.ubc.ca, zkelan@math.ubc.ca}


\end{document}